\documentclass[12pt]{amsart}
\usepackage{bbm}
\usepackage{amsfonts}
\usepackage{mathrsfs}
\usepackage{hyperref}
\usepackage{amssymb}
\usepackage{graphicx,color,tikz}
\usepackage{CJK}
\usetikzlibrary{mindmap}
  \newtheorem{thm}{Theorem}[section]
 
\newtheorem{prop}[thm]{Proposition}
 
\newtheorem{conjecture}[thm]{Conjecture}
\newtheorem{definition}[thm]{Definition}

  \def\f{\mathbb{F}}
  \def\F{\mathbb{F}}
 
\def\R{\mathbb{R}}
\def\Q{\mathbb{Q}}
\def\Z{\mathbb{Z}}

\def\ord{\mathrm{ord}}

 \newcommand{\md}[1]{\,(\mathrm{mod} \,\, #1)}

\begin{document}

\title[Newton Polygons of L-functions Associated to Deligne Polynomials]
  {Newton Polygons of L-functions Associated to Deligne Polynomials}

\author{Jiyou Li}
\address{Department of Mathematics, Shanghai Jiao Tong University, Shanghai, P.R. China}
\email{lijiyou@sjtu.edu.cn}
%


\maketitle \numberwithin{equation}{section}
 \allowdisplaybreaks

\begin{abstract}
A conjecture of Le says that the Deligne polytope $\Delta_d$ is generically ordinary if $p\equiv 1\ (\!\bmod \ D(\Delta_d))$, where
$D(\Delta_d)$ is a combinatorial constant determined by $\Delta_d$.
In this paper a counterexample is given to show that the conjecture
is not true in general.
\end{abstract}

\section{Introduction}

       Let $\f_q$ be the finite field of $q$ elements of characteristic $p$.      Let $\psi$ be a fixed nontrivial additive character  on $\f_q$.
     For a Laurent polynomial $f\in\F_{q}\left[x_1^{\pm1}, \ldots, x_n^{\pm1} \right]$,  define its associated $L$-function by
 \[ \label{Lfunction} L^*(f,T) = \exp \left( \sum_{k=1}^{\infty} \sum_{x \in {(\f_{q^k})^*}} \psi ({f(x)}) \frac{T^k}{k} \right). \]
 A celebrated theorem of Dwork-Bombieri-Grothendieck implies that $L^*(f,T)$ is a rational function in $\Q[\zeta_p](x)$ \cite{DW62}, \cite{GR64}.
 Furthermore,   Deligne's theorem on the Riemann hypothesis \cite{De81} gives very nice descriptions of the Archimedean absolute values of its zeroes and poles.

Adolphson and Sperber \cite{AS89} show that if the Laurent polynomial $f$ satisfies a non-degenerate geometric condition, then  $L^*(f,T)^{(-1)^{n-1}}$ is indeed a polynomial in
$\Z[\zeta_p][x]$ of degree $n!\text{V}(\Delta(f))$.
 Here   $\Delta(f)$ is the {\it Newton polytope} (see \cite{GKZ94} for details) of $f$, defined to be the convex closure in $\R^n$ generated by the origin and all exponents of the nonzero monomials (viewed as lattice points in $\R^n$), and $V(\Delta(f))$ is the volume.

  Since the Archimedean absolute values of the zeroes of $L^*(F,T)^{(-1)^{n-1}}$ are determined explicitly,  it remains to study the most intriguing problem of determining    the non-Archimedean absolute values. Equivalently, it suffices to
 study the slopes of  the Newton polygon of $L^*(f,T)^{(-1)^{n-1}}$,  which is defined as follows.
\begin{definition}Let $g(x) =\sum_{i=0}^d c_i x^i \in 1 + x \Z[\zeta_p]$. The ($q$-adic) Newton polygon of $g(x)$, denoted by $NP(f)$, is defined to be the lower convex hull of the points $(i, \ord_qc_i)$ in the plane $\mathbb{R}^2$. An example is given in Figure 1.
\end{definition}
  Unfortunately, the computation of the Newton polygon of a given
  L-function is very hard in general.
  However, as shown by the Dwork theory, in many cases the combinatorial geometrical properties of the Newton polytope may give very nice algebraic properties of $f$.
  For instance, when $f$ is non-degenerate, Adolphson and Sperber \cite{AS89} show that  the Newton polygon of
    $L^*(f,T)^{(-1)^{n-1}}$ lies above a topological lower bound
     called Hodge polygon, which will be briefly introduced as follows.

 Let $\Delta$ be an $n$-dimensional convex integral polytope in $\mathbb{R}^n$. Let $C(\Delta)$ be the cone generated by $\Delta$ and the origin.  For a vector $u\in\R^n$, $w(u)$ is the smallest nonnegative real number $c$ satisfying $u \in c \Delta$.  If such $c$ does not exist, then we define $w(u) = \infty$.

  For a  co-dimension $1$ face $\delta$ not containing $0$ in $\Delta$, let $\sum_{i=1}^n e_i X_i=1$ be the equation of the hyperplane containipng $\delta$ in $\R^n$.  Since the coordinates of vertices spanning $\Delta$ are integers, the coefficients $e_i$ are uniquely determined rational numbers.  Let $D(\delta)$ be the least common denominator of $\{e_i,  1 \leq i \leq n\}$ and $D=D(\Delta)$ be the least common denominator of all $D(\delta)$, where $\delta$ runs over all the co-dimension $1$ faces of $\Delta$ which do not contain the origin.

  For an integer $k$, let
$W_{\Delta}(k) = \# \{u \in \Z^n\  | \ w(u) = \frac{k}{D}\}.$
Let $$H_{\Delta}(k) = \sum_{i=0}^{n} (-1)^i {n \choose i} W_{\Delta}(k-iD).$$ 
%
%
%
%
%
%

\begin{definition}
\label{fig3}
The Hodge polygon $HP(\Delta)$ of $\Delta$ is defined to be the lower convex polygon in $\R^2$ with vertices
$$\left( \sum_{k=0}^{m} H_{\Delta}(k), \frac{1}{D} \sum_{k=0}^m k H_{\Delta}(k) \right), 0\leq m\leq n D.$$
\end{definition}
Equivalently, $HP(\Delta)$ is the polygon starting from the origin with a side of  slope $k/D$ with horizontal length $H_{\Delta}(k)$ for each integer $0 \leq k \leq nD$.

For a finite integral polytope $\Delta \subset \R^n$, let $\mathcal{M}_p(\Delta)$ be the set of non-degenerate $f$ over $\overline{\f_p}$ with $\Delta(f)=f$.
To study how $NP(f)$ varies for $f \in \mathcal{M}_p(\Delta)$,
 by the Grothendieck specialization theorem \cite{WAN00},
the generic Newton polygon is given by
\[GNP(\Delta,p): = \inf_{f \in \mathcal{M}_p(\Delta)} NP(f).\]

   Some work of computing $GNP(\Delta,p)$
   in one-dimensional cases has been done \cite{ZHU03} and in certain two variable cases in \cite{NIU12}.

In \cite{AS89} Adolphson and Sperber proved that for $f \in \mathcal{M}_p(\Delta)$, the graph of $NP(f)$ lies above the graph of $HP(\Delta)$ with the same endpoints, which can be stated as the following inequality $NP(f) \geq HP(\Delta).$ See Figure 1 for an example.

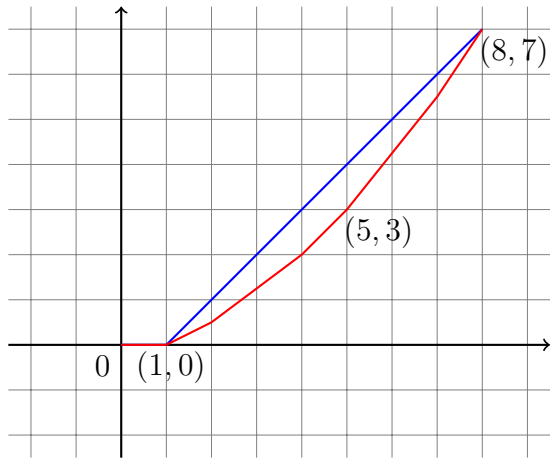
\begin{figure}
  \centering
\begin{tikzpicture}[scale=0.6]
\draw[step=1cm,gray,very thin] (-2.5, -2.5) grid (9.5,7.5);
\draw[->,thick] (-2.5,0) -- (9.5,0);
\draw[->,thick] (0,-2.5) -- (0,7.5);
\draw [thick, blue](0,0)--(1,0)--(8,7);
\draw [thick,red] (0,0)--(1, 0) -- (2, 1/2) -- (4, 2) -- (5, 3)--(7, 5.5)--(8, 7);
\draw (0,0) node[anchor=north east] {$0$};
\draw (1,0)+(0.1,-0.5) node {$(1,0)$};
\draw (5,3)+(0.7,-0.5) node {$(5,3)$};
\draw (8,7)+(0.7,-0.5) node {$(8,7)$};
\end{tikzpicture}
\caption{The Newton polygon (blue) and the Hodge polygon (red) of $ f(x_1, x_2) = x_1x_2^3+x_1^3x_2+x_1x_2$.}
\label{fig4}
\end{figure}

Combining these inequalities we have
\begin{prop}
For every prime $p$ and every $f \in \mathcal{M}_p(\Delta)$,
$$NP(f) \geq GNP(\Delta,p) \geq HP(\Delta).$$
\end{prop}

  A natural question is then to decide  when the second
inequality holds. For this purpose, we have the following definitions.

\begin{definition}
	A Laurent polynomial $f$ is called ordinary if NP($f$) = HP($\Delta$).
\end{definition}

\begin{definition}
 The family $\mathcal{M}_p(\Delta)$ is called {\it generically ordinary} if $GNP(\Delta,p)= HP(\Delta)$.
\end{definition}

 It is known that  a necessary condition for $GNP(\Delta,p)=HP(\Delta)$
 is $ p\equiv 1 \md{D(\Delta)}$.  Adolphson and Sperber \cite{AS87} conjectured that the converse is also true.
 \begin{conjecture}[Adophson-Sperber]
\label{AS}
 If $ p\equiv 1 \md{D(\Delta)}$, then $\mathcal{M}_p(\Delta)$ is  { generically ordinary}.
\end{conjecture}
This is a generalization of a conjecture of Dwork \cite{DW73} and Mazur \cite{MA72}.  Wan showed in \cite{WAN93} that Conjecture \ref{AS} is true for $n\leq 3$.  Furthermore,  he proved the following theorem.
\begin{thm}[Wan]
\label{theoremD}   For any $n\geq 4$, there are explicitly constructed counterexamples for the A-S conjecture.
However, a weaker version holds: there is an effectively computable integer $D^*(\Delta)$ such that if $p \equiv 1 \md{D^*(\Delta)}$, then $GNP(\Delta,p) = HP(\Delta)$.
\end{thm}
 Roughly speaking, one of the most important techniques of Wan's proof is a decomposition theorem which decomposes $\Delta$ into small nice pieces and reduce the problem to small problems which are easier to be checked (see Figure 2 for an illustration). Then  Wan developed several very useful decomposition theorems including  the star decomposition, the boundary decomposition  \cite{WAN93} and the collapsing decomposition \cite{WAN04}.

Suppose $\Delta=\Delta(f)$ and thus we can write $f$ as
\[f=\sum_{I\in \Delta \cap Z^n} a_I x^I.\]
Similarly, for a face $\delta\subseteq\Delta$, define
\[f^{\delta}=\sum_{I\in \delta \cap Z^n} a_I x^I.\]

\begin{thm}[Wan's facial decomposition theorem]
Suppose $f$ is non-degenerate (See \cite{AS89} for the definition).
Let $\{\delta_i, 1\leq i\leq s\}$ be the set all co-dimension 1 faces of $\Delta(f)$.
Then $f$ is ordinary if and only if each $f^{\delta_i}$ is ordinary.
\end{thm}
%

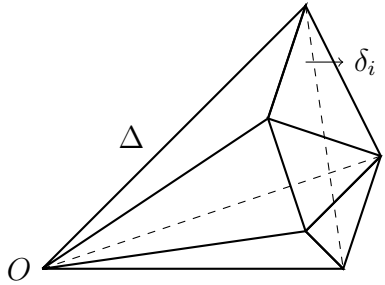
\begin{figure}
  \centering
\begin{tikzpicture}[scale=0.5]
 \draw [thick] (0, 0) -- (8, 0) -- (7,1)--(0,0);
   \draw [thick] (0,0)--(7,7)-- (6, 4) -- (0,0);
  \draw [thick] (7,1)-- (6, 4) -- (9,3)--(7,1);
    \draw [thick] (6,4)--(7,7) -- (9,3);
    \draw [thick] (9, 3) -- (8, 0) -- (7,1)--(9,3);
\draw (0,0) node[anchor=east] {$O$};
\draw (3,3.5) node[anchor=east] {$\Delta$};
 \draw [dashed] (0,0) -- (9,3);
  \draw [dashed] (7,7) -- (8,0);
  \draw[->,thin] (7, 5.5) -- (8,5.5) node[anchor=west]{$\delta_i$};
 \end{tikzpicture}
\caption{A facial decomposition for a 3 dimensional polytope.}
\label{fig5}
\end{figure}

In particular, it turns out that if all $f^{\delta_i}$'s are diagonal, then we have effective methods to determine if they are ordinary.
Since the cases for diagonal Laurent polynomials are very easy to deal with, this decomposition theorem will immediately give good nontrivial applications.

 Conjecture \ref{AS} is true in many important cases as showed by Wan \cite{WAN93} and more interesting examples are yet to be classified. Motivated by this,  Le \cite{LE12} studied an important class of Laurant polynomials, in which $f$ is given by
 \begin{equation}f(x_0, x_1, \ldots, x_n)=x_0 h(x_1, x_2, \ldots, x_n)+g(x_1, x_2, \ldots, x_n)+1/x_0,
\label{2.1}
\end{equation}
  where $\deg h=d$,  $\deg g<d/2$, and  $h$ is an $n$-variable Deligne polynomial, which  will be introduced in Section 2.
We are interested in the A-S conjecture for this family of polynomials.
  \begin{conjecture}
 Let $f$ be a generic polynomial given in \ref{2.1}  and let $\Delta=\Delta(f)$. Then
$GNP(\Delta,p) = HP(\Delta)$ if and only if $p \equiv 1 \md{D(\Delta)}$.
\end{conjecture}

Recall $\mathcal{M}_p(\Delta)$ is the set of non-degenerate $f$ over $\overline{\f_p}$ with $\Delta(f)=f$.
 Let $\widetilde{\mathcal{M}}_p(\Delta)\subseteq \mathcal{M}_p(\Delta)$ be the set of non-degenerate $f$ of shape \ref{2.1} with $\Delta(f)=\Delta$, and let
 \[\widetilde{GNP}(\Delta,p): = \inf_{f \in \widetilde{\mathcal{M}}_p(\Delta)} NP(f).\]
 Since clearly $\widetilde{GNP}({\Delta},p) \geq {GNP}({\Delta},p)\geq HP(\Delta),$ a slightly stronger conjecture arises.
\begin{conjecture}[Le, 2013]
$\widetilde{GNP}({\Delta},p) = HP(\Delta)$ if and only if $p \equiv 1 \md{D(\Delta))}$.
\label{Con1.10}
\end{conjecture}
Le proved this conjecture for odd $d$ and this implies that the A-S conjecture
for this family holds for odd $d$.   He then conjectured
that it is also true for even $d$.

In this paper we give a counterexample to this conjecture.
\begin{thm}
Le's conjecture is false in general.
\end{thm}
Note that the reason that Le's conjecture fails is because
when $d$ is even, the family polynomials  $\widetilde{\mathcal{M}}_p(\Delta)$ studied in this paper and the family of polynomials $\mathcal{M}_p(\Delta)$ are not identical. In other words, the family Le considered is slightly smaller and thus is not a universal family.
 A natural question is to determine whether  Le's conjecture
is true when the family of $f$ defined by \ref{2.1} becomes larger, i.e., the condition
"$\deg g<d/2$" is  replaced by "$\deg g\leq d/2$".

\section{ A counterexample to Le's conjecture}

A polynomial $f(x)=f(x_1, x_2, \ldots, x_n)$ in $\f_q[x_1, x_2, \ldots, x_n]$ is called a  {\it Deligne polynomial} if
      its leading homogeneous form $f_d$ defines a smooth projective hypersurface in the projective space $\mathbb{P}^{n-1}$, where $d$ is the degree of $f$.    Deligne polynomials were extensively studied by Deligne \cite{De81},  Katz \cite{KA07}, Browning and Heath-Brown \cite{BH09},  Le \cite{LE12}, Fu and Wan \cite{FW20}.

Deligne shows that if $f$ is a Deligne polynomial, then  $L^*(f,T)^{(-1)^{n-1}}$ is polynomial
of degree $(d-1)^n$ and this fact immediately gives a fundamental estimate
\[|\sum_{x \in {\f_{q^k}}} \psi ({f(x)})|  \leq (d-1)^n q^{\frac k2}.\]

Motivated by the study of exponential sums, it is natural to consider a general class of Laurent polynomials of the following form
\[f(x_0, x_1, \ldots, x_n)=x_0 h(x_1, x_2, \ldots, x_n)+g(x_1, x_2, \ldots, x_n)+1/x_0, \]
  where $\deg h=d$,  $\deg g<d/2$, and  $h$ is a Deligne polynomial in $n$ variables.

From now on, $\Delta$ is always used to denote the Newton Polytope of $f$ and  $e_0, e_1, \ldots, e_n$ be the standard basis for $\R^{n+1}$.
 Let $\delta_d$ be the polytope spanned by $-e_0, e_0+de_1, \ldots, e_0+de_n$, and  $\delta_d'$ spanned by $e_0,e_0+de_1, \ldots, e_0+de_n$.
  They are the only two  co-dimension $1$ faces of $\Delta$ that do not contain the origin.  Let $\Delta_d$($\Delta_d'$, respectively) be the convex hull of $\delta_d$($\delta_d'$, respectively) and the origin.
  Please see Figure 3 for an example.


%
%

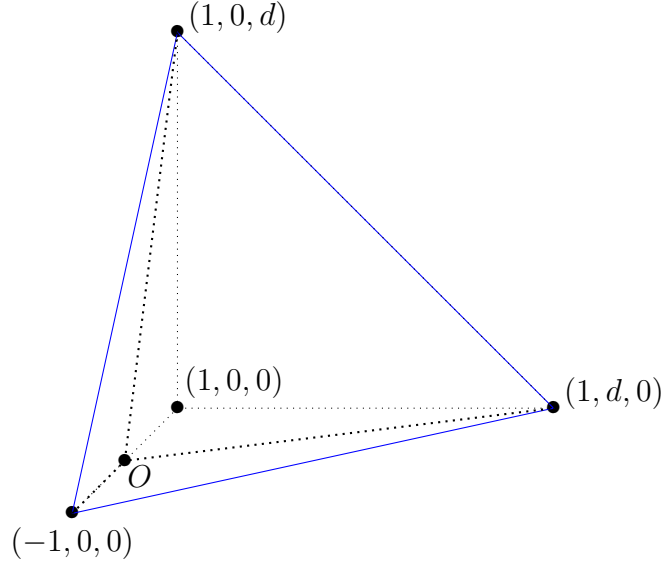
\begin{figure}[t]
\begin{center}
\begin{tikzpicture}
\draw(0,0) node {$\bullet$}; 
\node (n2) at (-0.7,-0.7) {$\bullet$}; 
\node (n1) at (-1.4,-1.4) {$\bullet$}; 
\node (v5) at (0,5) {$\bullet$}; 
\node(v0) at (5,0) {$\bullet$}; 
\draw[blue] (-1.4,-1.4) -- (0,5) -- (5,0) -- (-1.4,-1.4);
\draw[thick,dotted] (0,5) -- (-0.7,-0.7) -- (5,0);
\draw[thick,dotted]  (-1.4,-1.4) -- (-0.7,-0.7) ;
\draw[dotted] (0,0) -- (0,5) -- (5,0)--(0,0);
\draw[dotted] (0,0) -- (-1.4,-1.4);
\draw (n1)+(0,-0.4) node {$(-1,0,0)$};
\draw (n2)+(0.2,-0.2) node {$O$};
 \draw (0,0)+(0.75, 0.35) node {$(1, 0, 0)$};
\draw (v5)+(0.8,0.2) node {$(1, 0, d)$};
\draw (v0)+(0.8,0.2) node {$(1,d,0)$};
\end{tikzpicture}
\end{center}
\caption{$\Delta \supseteq \Delta_d \supseteq \delta_d$ ($\delta_d$ is the blue triangle)}
\label{fig6}
\end{figure}

Since the defining equation of $\delta_d$ is
\[ X_1-\frac{2}{d}X_2-\frac{2}{d}X_3-\cdots-\frac{2}{d}X_{n+1}=1,\]
one computes that $$D(\Delta_d)=
\begin{cases}
d &  d \mathrm{~is~odd;}\\
\frac{d}{2} & d \mathrm{~is~even}.
\end{cases}$$

For simplicity we denote $D(\Delta_d)$ by $D$ .
The Hodge polygon $HP(\Delta_d)$ can be determined explicitly. We would like to know if $\Delta_d$ satisfies Conjecture \ref{AS}.

 Let $\widetilde{\mathcal{M}}_p(\Delta_d)\subseteq \mathcal{M}_p(\Delta_d)$ be the set of non-degenerate $f$ of shape \ref{2.1} with $\Delta(f)=\Delta_d$.
 Let
 \[\widetilde{GNP}(\Delta_d,p): = \inf_{f \in \widetilde{\mathcal{M}}_p(\Delta_d)} NP(f).\]
Note that  $\widetilde{GNP}({\Delta_d},p) \geq {GNP}({\Delta_d},p)\geq HP(\Delta_d)$.

  By Wan's facial decomposition theorem,  Le's conjecture is then equivalent to
  \[\widetilde{GNP}({\Delta_d},p) = HP(\Delta_d) \Longleftrightarrow p \equiv 1 \md{D}.\]
  Applying the regular decomposition theorem, Le \cite{LE12} proved that
\begin{thm}
If $p \not\equiv 1 \md{D}$ then $GNP(\Delta_d, p)$ lies strictly above $HP(\Delta_d)$.
\label{prop_nonordinary}
In particular, for any odd $d$, $p \equiv 1 \md{D}$ if and only if  $\widetilde{GNP}(\Delta_d,p) = HP(\Delta_d)$.
\end{thm}
This implies that A-S conjecture for the family of polynomials (\ref{2.1}) holds for odd $d$.   Le then conjectured
that it also holds for even $d$. We now give a counterexample to this conjecture.

\begin{thm}
Le's conjecture is false in general.
\end{thm}

\begin{proof}
   Recall $\widetilde{\mathcal{M}}_p(\Delta_d)\subseteq \mathcal{M}_p(\Delta_d)$ be the set of non-degenerate $f$ of shape \ref{2.1} with $\Delta(f)=\Delta_d$. It suffices to prove
   $$\widetilde{GNP}(\Delta_d,p) \ne HP(\Delta_d).$$
 Let  $f$ be a polynomial in $\widetilde{\mathcal{M}}_p(\Delta_d)$.
 If we define
\[f^{\delta_d}=x_0 h(x_1, x_2, \ldots, x_n)+1/x_0, \]
and
\[f^{\delta_d'}=x_0 h(x_1, x_2, \ldots, x_n), \]
then equivalently the facial decomposition theorem implies that
     $f$ is generically ordinary if and only if both  $f^{\delta_d}$ and $f^{\delta_d'}$ are generically ordinary.
Since Le has already showed that  $f^{\delta_d'}$ is generically ordinary, his conjecture is then equivalent to the statement that the polynomial $f^{\delta_d}$ is also generically ordinary.

\begin{figure}[t]
\begin{center}
\begin{tikzpicture}
\node (v5) at (0,3) {$\bullet$}; 
 \node (v3) at (1.5,0) {$\bullet$}; 
 \node(v0) at (3,3) {$\bullet$}; 
  \node(v1) at (0.6,3) {$\bullet$}; 
  \node(v2) at (0.75,1.5) {$\bullet$}; 
  \node(v4) at (2.25,1.5) {$ \small\circ$}; 
\draw[] (0,3) -- (1.5, 0) -- (3,3)--(0,3);
  \draw (v5)+(-0.8,0.3) node {$(1, 0, d)$};
 \draw (v0)+(0.8,0.3) node {$(1,d,0)$};
    \draw (v2)+(-0.9,0) node {$(0, 0, \frac d2)$};
  \draw (v3)+(0.3,-0.4) node {$(-1,0,0)$};
    \draw (v1)+(0.8,0.3) node {$(1,1,d-1)$};
   \draw (v4)+(0.9,0) node {$(0,\frac d2,0)$};
\end{tikzpicture}
\end{center}
\caption{ A face $\tau\subseteq\delta_d$  for  $n=2$}
\label{fig7}
\end{figure}
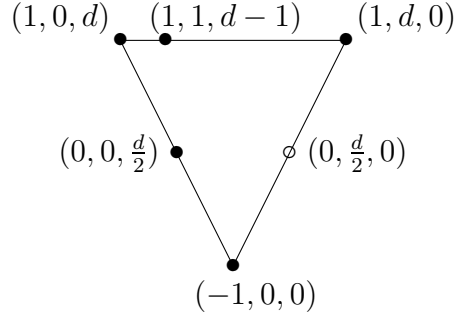


Suppose now $n=2$, $4\mid d$ and  $p\equiv \frac d2+1\md{d}$.
Let $\tau$ be a 1-dimensional face spanned by $(1, d, 0)$ and $(-1, 0, 0)$ on the two dimensional face $\delta_d$ (see Figure 4 for the example).
Its matrix $M$ is given by
\[M=\begin{bmatrix}
-1 & 1 \\
0 & d
 \end{bmatrix}.\]
 Clearly $\det M=-d$.

Let $r=(r_1, r_2)$.
Solving the system of linear equations \[
M\left( \begin{matrix} r_1 \\ r_2 \end{matrix}\right) \equiv 0 \md{1}, r_i \in \mathbb{Q}\cap[0, 1),\]
 we have
 \[r_1=r_2=\frac kd (0\leq k\leq d-1).\]
 This gives  $w(r)=\frac {2k}d$ and  hence $r$ is stable for even $k$.

 However, since  $p\equiv \frac d2+1\md{d}$, $\{pr_1\}=\{pr_2\}=
 \{\frac {(1+d/2)k}{d}\}$, for odd $k$ and $k<\frac d2$ we have
 $\{pr_1\}=\{pr_2\}= \frac kd+\frac 12$.  This shows that $w(r)$ is not stable for odd $k$ and $k<\frac d2$ and hence $\delta$ is not generically ordinary.

%
%

\begin{figure}[t]
\begin{center}
\begin{tikzpicture}
\draw [fill=black](-0.7,-0.7) circle (0.05cm);
\draw [fill=black](-1.4,-1.4) circle (0.05cm);
\node (n1) at (-1.4,-1.4) {};
\node(v5) at (1.8,-0.7) {};
\draw [fill=black](1.8,-0.7) circle (0.05cm);
\node (v4) at (0,5) {};
\draw [fill=black](0,5) circle (0.05cm);
 \node(v3) at (1.25,3.75) {};
 \draw [fill=black](1.25,3.75) circle (0.05cm);
\node(v2) at (2.5,2.5) {};
\draw [fill=black](2.5,2.5) circle (0.05cm);
\node(v1) at (3.75,1.25) {};
\draw [fill=black](3.75,1.25) circle (0.05cm);
\node(v0) at (5,0) {};
\draw [fill=black](5,0) circle (0.05cm);
\draw [thick, dotted](-0.7,-0.7)--(2.5,2.5);
\draw (-1.4,-1.4)--(2.5,2.5);
\draw [thick, dotted](-0.7,-0.7)--(1.25,3.75);
\draw (-1.4,-1.4)--(1.25,3.75);
\draw [thick, dotted](-0.7,-0.7)--(3.75,1.25);
\draw (-1.4,-1.4)--(3.75,1.25);
\draw (-1.4,-1.4) -- (0,5) -- (5,0) -- (-1.4,-1.4);
\draw[thick,dotted] (0,5) -- (-0.7,-0.7) -- (5,0);
\draw[thick,dotted]  (-1.4,-1.4) -- (-0.7,-0.7) ;
 \draw (n1)+(0,-0.4) node {$(-1,0,0)$};
\draw (v4)+(0.8,0.2) node {$(1,0,4)$};
\draw (v3)+(0.8,0.2) node {$(1,1,3)$};
 \draw (v2)+(0.8,0.2) node {$(1,2,2)$};
\draw (v1)+(0.8,0.2) node {$(1,3,1)$};
\draw (v0)+(0.8,0.2) node {$(1,4,0)$};
\draw[->,thin] (3.5,0.2) -- (4,-0.8) node[anchor=west] {$\Pi_1$};
\draw (v5)+(0.8,0.2) node [anchor=north] {$(0,2,0)$};
 \end{tikzpicture}
\end{center}
\caption{A boundary decomposition for the polytope $\Delta_d$, $d=4, n=2$}
\label{fig7}
\end{figure}
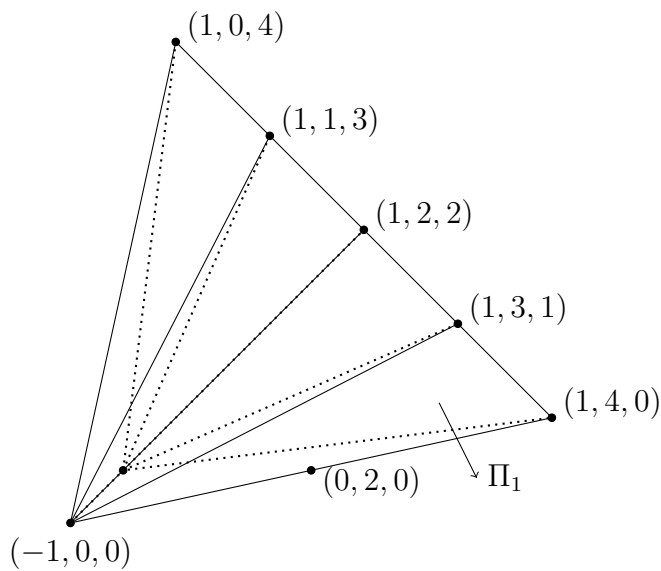
Let $\Pi_1$ be  the sub polytope spanned by $(0, 0, 0), (-1, 0, 0), (1, d, 0)$ and $(1, d-1, 1)$ (please see the figure below for an illustration). Since its face $\tau$ is not generically ordinary,  $\Pi_1$ is also  not generically ordinary.
 By the boundary decomposition theorem (\cite{WAN93}, Theorem 5.1, see also Figure 5 for the graph) $\Delta_d$ is not generically ordinary. The proof is complete.


\end{proof}

{\bf Acknowledgement.}   The author wishes to thank Professor  Daqing Wan for his instructive suggestions.


\begin{thebibliography}{10}

\bibitem{AS87} A. Adolphson and S. Sperber, \emph{Newton Polyhedra and the degree of the $L$-function associated to an exponential sum}.
 Invent. Math., 555--569, 1987.

\bibitem{AS89}
A. Adolphson and S. Sperber,
\emph{Exponential sums and Newton polyhedra: Cohomology and estimates}.
Ann. Math., 367--406, 1989.




\bibitem{BH09} T. Browning and R. Heath-Brown,
\emph{Integral points on cubic hypersurfaces},
Analytic number theory, Cambridge University Press, 75--90, 2009.

\bibitem{De81} P. Deligne, \emph{La conjecture de Weil. II. (French) [[Weil's conjecture. II]]}, Inst. Hautes \'{E}ptudes Sci. Publ. Math. No. 52 (1980), 137--252.

\bibitem{DW62} B. Dwork, \emph{On the zeta function of a hypersurface}, Publ. Math. IHES, 12 (1962), 5--68.

\bibitem{DW73} B. Dwork, \emph{Normalized period matrices}, Ann. of Math., 98 (1973), 1--57.

\bibitem{FW20} L. Fu and D. Wan, \emph{On Katz's (A,B)-exponential sums}, arXiv:2003.08796.




\bibitem{GKZ94}
I. Gelfand, M. Kapranov, and A. Zelevinsky,
\emph{\em {Discriminants,Resultants, And Multidimensional Determinants}}. Birkhauser Verlag Ag (Switzerland), 1994.

\bibitem{GR64} A. Grothendieck,
\emph{Formule de Lefschetz et rationalit\'e des fonctions $L$},
S\'eminaire Bourbaki, Vol. 9, Exp. No. 279, 41--55, Soc. Math. France, Paris, 1995.


\bibitem{KA07} N. Katz, \emph{On a question of Browning and Heath-Brown},  Analytic number theory, 267--288, Cambridge Univ. Press, Cambridge, 2009.
%


\bibitem{LE12} P. Le, \emph{Regular decomposition of ordinarity in generic exponential sums}, J. Number Theory 133 (2013), 2648--2683.

\bibitem{MA72} B. Mazur, \emph{Frobenius and the Hodge Filtration},  Bull. A.M.S., 78 (1972), 653--667.

\bibitem{NIU12} C. Niu, \emph{Generic Exponential Sums Associated with Polynomials of Degree 3 in   Two Variables},  Acta Mathematica Sinica (English Series),  2301--2312, 2012.

\bibitem{Sage} SageMath, \emph{the Sage Mathematics Software System (Version 8.1)}, The Sage Developers, 2018, https://www.sagemath.org.


\bibitem{WAN93} D. Wan,  \emph{Newton polygons of zeta functions and L-functions}, Ann. of Math, 2 (1993) 249--293.

\bibitem{WAN00} D. Wan, \emph{Higher rank case of Dwork's conjecture}.
J. Amer. Math. Soc., 13 (2000) 807--852.

\bibitem{WAN04}  D. Wan, \emph{Variation of p-adic Newton polygons for L-functions of exponential   sums},  Asian Journal of Mathematics, 8 (2004), 427--472.


\bibitem{ZHU03} H. Zhu,  \emph{$p$-adic variation of $L$-functions of one variable exponential sums, I}, Amer. J. Math, 125 (2003), 669--690.

\end{thebibliography}
\end{document}